\documentclass[11pt]{amsart}

\usepackage {amsfonts}
\usepackage[english]{babel}
\usepackage[cp1251]{inputenc}

\usepackage{amsmath}
\usepackage{amssymb}

\def\g{\gamma}
\def\G{\Gamma}
\def\d{\delta}

\def\r{\rho}
\def\p{\varphi}
\def\e{\varepsilon}
\def\l{\lambda}
\def\L{\Lambda}
\def\o{\omega}
\def\O{\Omega}
\def\t{\theta}

\def\s{\sigma}

\def\D{\mathcal D}
\def\S{\mathcal S}

\def\R{{\mathbb R}}
\def\C{{\mathbb C}}
\def\N{{\mathbb N}}
\def\Z{{\mathbb Z}}
\def\SH{\mathcal S_H}
\def\Re{\text Re\,}
\def\Im{\text Im\,}

\DeclareMathOperator{\supp}{supp}

\newtheorem{Th}{Theorem}
\newtheorem*{Pro}{Proposition}

\begin{document}

\title{Some properties of Fourier quasicrystals and measures on a strip}
\author{Sergii  Favorov, \"Ozkan De\v{g}er}

\address{Sergii Favorov,
\newline\hphantom{iii}  V.N.Karazin Kharkiv National University, Department of Pure Mathematics
\newline\hphantom{iii} Svobody sq., 4, Kharkiv, Ukraine 61022}
\email{sfavorov@gmail.com}

\address{\"Ozkan De\v{g}er,
\newline\hphantom{iii} Istanbul University
\newline\hphantom{iii}  Department of Mathematics, Faculty of Science, Istanbul, T\"urkiye, 34134}
\email{ozdeger@istanbul.edu.tr}

\maketitle{\small
\begin{quote}
\centerline {\bf Abstract.}

 We extend certain  results of the theory of Fourier quasicrystals on the real line to the case of a horizontal strip of finite width. We define the Fourier transform for measures on a strip that is
  a natural generalization of the  Fourier transform for measures on the line. For positive or translation bounded measures $\mu$ on a strip with the Fourier transform of the form $\hat\mu=\sum_{\gamma\in\Gamma}b_\gamma\delta_\gamma$ we prove that  the measure  $\nu=\sum_{\gamma\in\Gamma}|b_\gamma|^2\delta_\gamma$ has the exponential growth. If for some $\eta>0$ the points of $\Gamma$ in every interval of length  $\eta$ are linearly independent over integers then the measure $\hat\mu$ also has the exponential growth.
\bigskip


\end{quote}
}

   \bigskip
\section{Introduction}\label{S1}
   \bigskip

 A complex measure $\mu$ in a Euclidean space  with discrete locally finite support  is called {\it crystalline},
 if $\mu$ is a tempered distribution and its Fourier transform in the sense of distributions $\hat\mu$
 is also a measure with discrete locally finite support; $\mu$ is called a {\it Fourier quasicrystal} if both measures $|\mu|$ and $|\hat\mu|$ are also tempered distributions.
 Here and below  $|\nu|(E)$ denotes the variation of the complex measure $\nu$ on the set $E$.

 Fourier quasicrystals and crystalline measures have been studied very actively in the recent years. Many works are devoted to  investigations of their properties (see, for example, the  papers \cite{M4,LT,AKV}). In fact,  Fourier quasicrystals are  forms of Poisson formulas, which were used, in particular, by D. Radchenko and M. Viazovska in \cite{RV}.

 The first nontrivial example of Fourier quasicrystals, which  is a  sum of  unit masses
\begin{equation}\label{a}
\mu=\sum_{\l\in\L}\d_\l,\quad \l\in\R,
\end{equation}
 was found by P.Kurasov and P.Sarnak \cite{KS}. In \cite{OU} A. Olevskii and A. Ulanovskii  established a 1-1 connection between such Fourier quasicrystals
 and sets of zeros of real-rooted exponential polynomials. In \cite{F6} this result was generalized   to the zero sets of real-rooted absolutely convergent Dirichlet series. A multidimensional version of the result by Olevskii and Ulanovskii was partially obtained in papers [1] and [15], whereas the case of an arbitrary measure on the line whose Fourier transform is a pure point measure was fully investigated in papers [8] and [9].

 Measures of the form \eqref{a} with support $\L$ on a horizontal strip of finite width appeared in \cite{F7}, where an analog of Olevskii and Ulanovskii's result was found. Also in \cite{F9}, the  technique of Fourier quasicrystals on a strip was used to provide a criterion for an entire almost periodic function to be a finite product of sines.

 Returning to the definition of a Fourier quasicrystal, it is natural to ask whether we need additional conditions on the variations
of the measures $\mu$ and $\hat\mu$ in order to prove that a crystalline measure is a Fourier quasicrystal.  The answer is positive;  in the paper \cite{F5} a crystalline measure on $\R$ was constructed
that is not a Fourier quasicrystal.  In the paper \cite{BF} some conditions were found for a crystalline measure to be a Fourier quasicrystal. In particular, it was proved that
if a  measure $\mu$ on $\R$ is positive or translation bounded and its Fourier transform $\hat\mu$ is a pure point measure
\begin{equation}\label{b}
\hat\mu=\sum_{\g\in\G}b_\g\d_\g,
\end{equation}
then the measure
$$
\nu=\sum_{\g\in\G}|b_\g|^2\d_\g
$$
is also translation bounded, and if  points of $\G$ in every interval of length $\eta$ are linearly independent over $\Z$, then  $\hat\mu$ is  also translation bounded.

The set $\G$ in \eqref{b} is called {\it the spectrum} of the measure $\mu$. Recall also that a measure on $\R$ is {\it translation bounded} if its variations on intervals of length $1$ are uniformly bounded.

In our paper we show that analogs of these results are not valid for measures on a horizontal strip of finite width. Nevertheless we prove that for non-negative or translation bounded measures $\mu$ on a strip
the measures $\nu$ as above and the measure $\hat\mu$ in the case of locally independent set $\G$ have at most exponential growth.
The precise formulations  are presented in  Section \ref{S2}  after giving necessary definitions. Here we also present an analogue of the Poisson formula. An example of calculating the Fourier transform of a measure on a strip is given in Section \ref{S3}. In Sections \ref{S4} and \ref{S5}, we give some auxiliary results. In Section \ref{S6}, we present the proofs of the main theorems.

\bigskip
\section{Necessary notions and formulations of the main results}\label{S2}
\bigskip

Let $\mu$ be a measure with support on $\SH:=\{z\in\C:\,|\Im z|\le H\}$. We will say that the measure $\mu$ is {\sl translation bounded} on $\SH$ if
\begin{equation}\label{tr}
  \sup_{s\in\R}|\mu|(\{z\in\SH:\,s\le\Re z\le s+1\})=C_2<\infty.
\end{equation} 
Next, for any measure $\mu$ on $\SH$, let us denote
$$
M_\mu(r):=|\mu|(\{z\in\SH:\,|z|\le r\}).
$$
We shall assume that
\begin{equation}\label{rho}
M_\mu(r)\le C_1\max\{1,r^\r\}
\end{equation}
for some positive constants $\r$ and $C_1$.

 Let a  function $\p$ belongs to the  space $\D$ of all $C^\infty$-functions on $\R$ with compact support. Denote by
 \begin{equation}\label{f}
   \hat\p^c(z):=\int_{\R}\p(t)e^{-2\pi i zt}dt,\qquad z\in\C.
\end{equation}
the entire extension of the usual Fourier transform
$$
   \hat\p(x):=\int_{\R}\p(t)e^{-2\pi i xt}dt,\qquad x\in\R.
$$
Clearly,
\begin{equation}\label{f1}
   \hat\p^c(z)=\widehat{(\p(t)e^{2\pi yt})}(x),\quad z=x+iy.
\end{equation}
It follows from the definition \eqref{f} that for all $m\in\N$ uniformly with respect to $|y|\le H$
\begin{equation}\label{m}
  |\hat\p^c(z)|\le C_3(\max\{1,|z|\})^{-m}.
\end{equation}
Let  a measure $\mu$ on $\SH$ satisfy \eqref{rho}. Integrating by parts, we get
\begin{equation}\label{i}
 \int_{z\in\SH}|\hat\p^c(z)||\mu|(dz)\le C_3 M_\mu(1)+C_3\int_{|z|>1}|z|^{-m}|\mu|(dz)\le C_3 m\int_1^\infty r^{-m-1}M_\mu(r)dr.
\end{equation}
Therefore, for $m>\rho$ the integral
\begin{equation}\label{i1}
 \int_{z\in\SH}\hat\p^c(z)\mu(dz)
\end{equation}
converges absolutely. So, {\it the $c$-Fourier transform} of a measure $\mu$  is well-defined by the equality
\begin{equation}\label{i2}
(\hat\mu^c,\p)=\int\hat\p^c(z)\mu(dz).
\end{equation}
It is easy to see that the definition of \eqref{f} generalizes to all functions $\p$ on $\R$ with compact support that have bounded derivatives up to order $m>\rho$.
Consequently, \eqref{m} and \eqref{i2} are also valid for such functions.

In the next section  we will  show that  the $c$-Fourier transform of the measure $\mu_0=\sum_{n\in\Z}\d_{n+i}$ is the measure $\hat\mu^c_0=\sum_{n\in\Z}e^{2\pi n}\d_n$.
Clearly,
$$
 \log \hat\mu^c_0(-r,r)=\log \sum_{|n|\le r}e^{2\pi n}\sim 2\pi r \qquad(r\to\infty).
$$
Hence the measure $\hat\mu^c_0$ cannot be translation bounded, and the exact analog of results in \cite{BF} is incorrect. To formulate
our generalization of these results we recall the following definition (\cite[Ch.1]{L}):

A measure $\nu$ on $\R$ has a type $\s$ with respect to the exponential growth if
$$
  \limsup_{r\to\infty}r^{-1}\log\nu(-r,r)=\s.
$$
In our paper we obtain the following results:

\begin{Th}\label{T1}
If a positive measure $\mu$ on $\SH$ satisfies \eqref{rho} and its $c$-Fourier transform $\hat\mu^c$ is also a measure, then $\mu$ is translation bounded.
\end{Th}

\begin{Th}\label{T2}
Let $\mu$ be a  non-negative measure on $S_H$ satisfying (3), and its $c$-Fourier transform $\hat\mu^c$ be a pure point measure of the form \eqref{b}.
Then the measure
\begin{equation}\label{sq}
 \nu=\sum_{\g\in\G}|b_\g|^2\d_\g
\end{equation}
has  the property
\begin{equation}\label{pr}
 \nu(t-1,t+1)=O(e^{4\pi H|t|})\qquad\text{as}\quad t\to\infty,
\end{equation}
and has the type  with respect to the exponential growth at most $4\pi H$.
The same statement is true for any translation bounded measure $\mu$ on $\SH$ with a pure point measure $\hat\mu^c$.
\end{Th}

\begin{Th}\label{T3}
Let $\mu$ be a non-negative measure on $S_H$ satisfying (3), and its $c$-Fourier transform $\hat\mu^c$ be a pure point measure of the form \eqref{b}.
If there exists $\eta>0$ such that  sets $\G\cap (t-\eta,t+\eta)$ are linearly independent over $\Z$ for all $t\in\R$,
then the measure $\hat\mu^c$ has the property
\begin{equation}\label{pr1}
 |\hat\mu^c|(t-1,t+1)=O(e^{2\pi H|t|})\qquad\text{as}\quad t\to\infty,
\end{equation}
and has the type  with respect to the exponential growth at most $2\pi H$. The same statement is true for any translation bounded measure $\mu$ on $\SH$
with a pure point measure $\hat\mu^c$  satisfying the condition of the theorem.
\end{Th}
Recall that the set $X\subset\R$ is linearly independent over $\Z$ if for any finite number of elements $t_1,\dots,t_n\in X$
and arbitrary $m_1,\dots,m_n\in\Z$ the equality $m_1t_1+\dots+m_nt_n=0$ implies $m_1=\dots=m_n=0$.

Note that the condition on $\G$ is satisfied   in the case $|t-t'|>\eta$ for all $t,\,t'\in\G,\ t'\neq t$. Thus, the above example of the measure $\mu_0$ shows that estimates \eqref{pr}
and \eqref{pr1} are optimal.
\medskip

In conclusion,  note that in the case
$$
\mu=\sum_{\l\in\L}a_\l\d_\l,\quad \L\subset\SH,\qquad \hat\mu^c=\sum_{\g\in\G}b_\g\d_\g,\quad \G\subset\R,
$$
we obtain an analogue of the Poisson formula
 $$
 \sum_{\g\in\G}b_\g\p(\g)=\sum_{\l\in\L}a_\l\hat\p^c(\l),
 $$
here $\hat\p^c$ is defined in \eqref{f}.

\bigskip
\section{An example of $c$-Fourier transform for a measure in a strip}\label{S3}
\bigskip

Let us find the $c$-Fourier transform of the measure $\mu_0 = \sum_{n\in\Z} \d_{n+i}$
on the strip $S_1=\{z:\,|\Im z|\le 1\}$.

Let $\p\in \D$ be a test function. Since the measure $\mu_0$ satisfy \eqref{rho} and the integral \eqref{i1} with $\mu=\mu_0$  converges  absolutely,  we can change the order of integration and summation and get
$$
(\hat{\mu}^c_0, \p) = \int_{\S_1} \hat\p^c(z) \mu_0(dz) =  \sum_{n\in\Z} \int_{S_1} \hat\p^c(z) \d_{n+i} = \sum_{n\in\Z} \hat\p^c(n+i).
$$
It follows from \eqref{f1} that $\hat\p^c(n+i)= \widehat{\left( \p(t) e^{2\pi t} \right)}(n)$, therefore, according to the classical Poisson formula,
$$
(\hat{\mu}^c_0,\p)=\sum_{n\in\Z}\widehat{\left( \p(t) e^{2\pi t} \right)}(n)=\sum_{n\in\Z}\p(n) e^{2\pi n}=\sum_{n\in\Z} e^{2\pi n}\d_n(\p),
$$
and $\hat{\mu}^c_0 = \sum_{n\in\Z} e^{2\pi n} \d_n$.

\bigskip
\section{Some propeties of the $c$-Fourier transform}\label{S4}
\bigskip

The next three equalities  for all $\p\in\D$ and $\mu$ satisfying \eqref{rho} follow immediately from the definition of $c$-Fourier transform:
\begin{equation}\label{p1}
\widehat{(e^{2\pi i\tau t}\p(t))^c}(z) = \hat{\p}^c(z-\tau), \qquad z\in\SH, \quad \tau\in\R,
\end{equation}
\begin{equation}\label{p2}
\hat\p^c_\tau(z) = e^{-2\pi i\tau z}\hat\p^c(z) \quad \text{for} \quad \p_{\tau}(t)=\p(t-\tau), \qquad z\in\SH, \quad \tau\in\R,
\end{equation}
\begin{equation}\label{p3}
\hat\mu_\tau^c(t) = e^{-2\pi i\tau t}\hat{\mu}^c(t) \quad \text{for} \quad \mu_{\tau}(E)=\mu(E-\tau), \qquad \tau\in\R.
\end{equation}
 Since for $s\in\R,\ z\in\SH$
$$
  |e^{2\pi isz}\hat\p^c(z)|\le e^{2\pi|s|H}|\hat\p^c(z)|\le C_3e^{2\pi|s|H}(\max\{1,|z|\})^{-m},
$$
we see that the same arguments as in \eqref{i} prove the estimate
\begin{equation}\label{p6}
\left|\int_{\SH}e^{2\pi isz}\hat\p^c(z)\mu(dz)\right|\le C_4 e^{2\pi|s|H},
\end{equation}
where the constant $C_4$ does not depend on $s$.

  If the measure $\mu$ is translation bounded, then  for any shift of the measure $\mu_\tau(E)=\mu(E-\tau)$ we have
 $$
|\mu_\tau|(\{z\in\SH:\,|z|\le r\})\le 4C_2\max\{1,r\},
$$
 where $C_2$ is the constant from \eqref{tr}.  We obtain from \eqref{p6}  with $m=2$
\begin{equation}\label{p7}
\left|\int_{\SH}e^{2\pi isz}\hat\p^c(z)\mu_\tau(dz)\right|\le C_5 e^{2\pi|s|H},
\end{equation}
where the constant $C_5$ does not depend on either $s$ or $\tau$.

We will also use the following proposition
\begin{Pro}
For any $m\in\N$ there is a  function $\Psi_m(t)$ with compact support such that its $c$-Fourier transform $\hat\Psi_m^c(z)$ satisfies \eqref{m}
and $\Re\hat\Psi_m^c(z)>0$ for all $z\in\SH$.
\end{Pro}
{\bf Proof}. Set $\t(t)=1/2$ for $|t|\le1$ and $\t(t)=0$ for $|t|>1$. The $c$-Fourier transform of the function $\t\star\t$ equals $\sin^2(2\pi z)/(2\pi z)^2$,
therefore it is non-negative for all $z\in\R$ and vanishes only at the points $z=k/2,\ k\in\Z\setminus\{0\}$. Consider the function
$$
   \psi(t)=e^{i\pi t/4}(\theta\star\theta)(t)+e^{-i\pi t/4}(\theta\star\theta)(t).
$$
It has compact support and $c$-Fourier transform
$$
   \hat\psi^c(z)=\sin^2(2\pi z-\pi/4)/(2\pi z-\pi/4)^2+\sin^2(2\pi z+\pi/4)/(2\pi z+\pi/4)^2,
$$
that is strictly positive on the real axis. The simple calculation shows that
\begin{gather*}
   \hat\psi^c(z)=\frac{i}{4}\left(e^{4\pi iz}-e^{-4\pi iz}\right)\left[\frac{1}{(2\pi z-\pi/4)^2}-\frac{1}{(2\pi z+\pi/4)^2}\right]\\
+\frac{1}{4}\left\{\frac{2}{(2\pi z-\pi/4)^2}+\frac{2}{(2\pi z+\pi/4)^2}\right\}.
\end{gather*}
The expression in the square brackets behaves like $O(z^{-3})$, and the expression in the braces behaves like $(\pi z)^{-2}$ as $z\to\infty$. Hence for $z\in\SH,\ z\to\infty$,
$$
  \hat\psi^c(z)=O(z^{-3})+\frac{1+o(1)}{4\pi^2z^2}=\frac{1}{4\pi^2(\Re z)^2}(1+o(1)).
$$
Define the convolution of $n\ge m/2$ functions $\psi(t)$ with $\Phi(t)$. Clearly, this function   has compact support. Also, its $c$-Fourier transform
$$
   \hat\Phi^c(z)=\left[\sin^2(2\pi z-\pi/4)/(2\pi z-\pi/4)^2+\sin^2(2\pi z+\pi/4)/(2\pi z+\pi/4)^2\right]^n
$$
satisfies \eqref{m} and  strictly positive on  the real axis. Therefore $\Re \hat\Phi^c(z)$ is strictly positive in a neighborhood of $\R$. We also have  for $z\in\SH,\ z\to\infty$
$$
 \hat\Phi^c(z)=\frac{1}{4^n\pi^{2n}(\Re z)^{2n}}(1+o(1)),
$$
hence there is $N<\infty$ such that $\Re\hat\Phi^c(z)>0$ for $z\in\SH,\ |\Re z|\ge N$. On the other hand, there is $\eta>0$ such that $\Re\hat\Phi^c(z)>0$ for $|\Im z|\le\eta,\ |\Re z|\le N$.
Consequently, for the compactly supported function $\Psi_m(t):=(H/\eta)\Phi(Ht/\eta)$  we obtain
$$
  \Re\hat\Psi_m^c(z)=\Re\hat\Phi^c(\eta z/H)>0
$$
for all $z\in\SH$.\qquad$\Box$
\bigskip

\bigskip
\section{Some properties of almost periodic functions}\label{S5}
\bigskip

The proofs of our theorems are based on some results of the classical theory of almost periodic functions
(see \cite{B}, \cite{Le}).

Recall that a continuous function $f(t)$ on the real line $\R$
is called {\sl almost periodic} if for any  $\e>0$ the set of its $\e$-almost periods
  $$
E_\e= \{\tau\in\R:\,\sup_{t\in\R}|f(t+\tau)-f(t)|<\e\}
  $$
is relatively dense, i.e., $E_\e\cap(t,t+L)\neq\emptyset$ for all $t\in\R$ and some $L=L(\e)$.

Each almost periodic function on $\R$ is bounded,  a finite sum or product of almost periodic  functions is also almost periodic,
and the uniform limit on $\R$ of almost periodic functions is almost periodic too. In particular,
every  absolutely convergent Dirichlet series
\begin{equation}\label{Dir}
   D(t)=\sum_{\o\in\O}a_\o e^{2\pi it\o},\qquad \sum_{\o\in\O}|a_\o|<\infty,
\end{equation}
with a countable $\O\subset\R$ is almost periodic. It is easy to check that the usual Fourier transform of $D(t)$ is the measure
$$
\hat D=\sum_{\o\in\O}a_\o \d_\o
$$
with a finite total mass.

For every almost periodic $f(t)$ and every $\o\in\R$ one considers
the Fourier coefficients $c_\o(f)\in\C$ as
$$
    c_\o(f)=\lim_{T\to\infty}\frac{1}{2T}\int_{-T}^Tf(t)e^{-2\pi it\o}dt.
$$
The set $\{\o:\,c_\o(f)\neq 0\}$ is at most countable.
It is easy to check that the uniform convergence of the series \eqref{Dir} implies $c_\o(D)=a_\o$ for all $\o\in\R$.
Next,  the Parseval identity
\begin{equation}\label{Par}
 \sum_{\o\in\R}|c_\o(f)|^2=\lim_{T\to\infty}\frac{1}{2T}\int_{-T}^T|f(t)|^2dt
\end{equation}
holds.

Note that by Y.Meyer (cf. \cite[Theorem 3.8]{M3}), if the Fourier transform   of the measure $f(t)dt$ for an almost periodic function  $f$ is also a measure, then
$$
  \sum_{|\o|<r}|c_\o(f)|<\infty, \quad \forall\,r<\infty,\qquad\text{and}\quad  \widehat{f(t)dt}=\sum_{\o\in\O} c_\o(f)\d_\o.
$$
 We will also use the Kroneker lemma (see \cite[Chapter II]{Le}):

Let $x_1,\dots,x_N$ be linearly independent over $\Z$ real numbers and $\t_1,\dots,\t_N$ be arbitrary real numbers. Then for every $\e>0$
there exist $t\in\R$ and $p_j\in\Z$, $j=1,\ldots,N$, such that
$$
|tx_j-\t_j-p_j|<\e,\quad j=1,\ldots,N.
$$

\bigskip
\section{Proofs of theorems}\label{S6}
\bigskip

{\bf Proof of Theorem \ref{T1}}. Let a measure $\mu\ge0$ be satisfied \eqref{rho} with some $\r<\infty$, and $\Psi_m$ be the function from the Proposition with $m>\r$.
Set $\eta=\inf_{z\in\SH:\,0\le\Re z\le 1}\Re\hat\Psi_m^c(z)$. Using \eqref{p3}, we get for any fixed $s\in\R$
\begin{gather*}
 \int_{z\in\SH:\,s\le\Re z\le s+1}\mu(dz)\le\eta^{-1}\int_{z\in\SH:\,0\le\Re z\le 1}\Re\hat\Psi_m^c(z)\mu_{-s}(dz) \le\eta^{-1}\int_{\SH}\Re\hat\Psi_m^c(z)\mu_{-s}(dz)\\
=\eta^{-1}\Re\int_{\SH}\hat\Psi_m^c(z)\mu_{-s}(dz)\le\eta^{-1}\left|\int_{\SH}\hat\Psi_m^c(z)\mu_{-s}(dz)\right|=\eta^{-1}\left|\int_\R\Psi_m(t)\hat\mu_{-s}^c(dt)\right|=\\
 \eta^{-1}\left|\int_\R \Psi_m(t)e^{2\pi ist}\hat\mu^c(dt)\right| \le\eta^{-1}\int_\R|\Psi_m(t)||\hat\mu^c|(dt).
\end{gather*}
Since $\Psi_m$ is compactly supported and $|\hat\mu^c|$ is a measure, we see that the last integral is finite, and the first integral is bounded by the independent on $s$ constant. \qquad$\Box$

{\bf Proof of Theorem \ref{T2}}. By Theorem \ref{T1}, it is enough to prove  Theorem \ref{T2} for  any translation bounded measure $\mu$.

Let $\p\in\D$ be a non-negative even function  such that $\p(t)\equiv1$ for $|t|<1$.  Note that $\hat\p^c(z)$ is also even, and for any $\tau\in\R$
$$
\hat\p^c(\tau-z)=\hat\p^c(z-\tau).
$$
 Using \eqref{p1} and the definition of the Fourier transform of measure on $\SH$, we get  for every $\tau\in\R$
\begin{eqnarray*}
 (\mu\star\hat\p^c)(\tau)=\int_{z\in\SH}\hat\p^c(\tau-z)\mu(dz) &=& \int_{z\in\SH}\widehat{(e^{2\pi i\tau t}\p(t))^c}(z)\mu(dz)\\
 = \int_\R \p(t)e^{2\pi i\tau t}\hat\mu^c(dt)  &=& \sum_{\g\in\supp\p}b_\g\p(\g)e^{2\pi i\tau\g}.
\end{eqnarray*}
Since $\hat\mu^c$ is a measure, we see that $\sum_{\g\in\supp\p}|b_\g|<\infty$  and the last sum converges absolutely. Therefore,
$(\mu\star\hat\p^c)(\tau)$ is an almost periodic function on $\R$. The above equality also implies (see Section \ref{S5}) that  the usual Fourier transform of this function  is the measure
$$
\widehat{\mu\star\hat\p^c}=\sum_{\g\in\supp\p}\p(\g)b_\g \d_\g.
$$
By Meyer's result, the Fourier coefficients of the function $(\mu\star\hat\p^c)(\tau)$ are
$$
   c_\g(\mu\star\hat\p^c)=b_\g\p(\g).
$$
Using Parseval's equality \eqref{Par}, we obtain that mass of the measure $\nu$ from \eqref{sq} on $[-1,1]$ can be estimated as follows
\begin{gather*}
  \nu[-1,1] = \sum_{|\g|<1}|b_\g|^2  \le \sum_{\g\in\R}|\p(\g)b_\g|^2= \sum_{\g\in\R}| c_\g(\mu\star\hat\p^c)|^2 \\
=\lim_{T\to\infty}\frac{1}{2T}\int_{-T}^T|(\mu\star\hat\p^c)(\tau)|^2d\tau\le\sup_{\tau\in\R}|(\mu\star\hat\p^c)(\tau)|^2=\sup_{\tau\in\R}\left|\int_{\SH}\hat\p^c(z)\mu_{-\tau}(dz)\right|^2.
\end{gather*}
 The measure  $\mu$ is translation bounded, hence it follows from \eqref{p7} with $s=0$ that
$$
  \left|\int_{\SH}\hat\p^c(z)\mu_{-\tau}(dz)\right|\le C_5.
$$
 If we replace $\p(t)$ with $\p(t-t')$ with $t'\in\R$ then by \eqref{p2} we have to replace $\hat\p^c(z)$ with  $e^{-2\pi it'z}\hat\p^c(z)$. By \eqref{p7},  we obtain
$$
   \nu[t'-1,t'+1]\le C_5^2e^{4\pi|t'|H}.
$$
We have proved \eqref{pr}. Then, each interval $(-r,r)$ can be covered by $\lfloor r\rfloor+1$ intervals of the form $(t-1,t+1)$ with $t<r$, where $\lfloor r\rfloor$ is an integer part of the number $r$. Therefore,
$$
\nu(-r,r)\le C_5^2(\lfloor r\rfloor+1)e^{4\pi rH},
$$
and
$$
\limsup_{r\to\infty}r^{-1}\log\nu(-r,r)\le 4\pi H.\qquad\Box
$$
\bigskip

{\bf Proof of Theorem \ref{T3}}.  By Theorem \ref{T1}, it is enough to prove  Theorem \ref{T3} for any translation bounded measure $\mu$.

Let $\p\in\D$ be such that $0\le\p(t)\le1$, $\supp\p\subset (-\eta,\eta)$,  $\p(t)\equiv1$ for $|t|<\eta/2$, and $C_5$ is the constant from \eqref{p7} that corresponds to this $\p$.
Suppose that \eqref{pr1} does not valid. Then there are  intervals $I_n=(t_n-\eta/2,t_n+\eta/2)$ such that $t_n\to\infty$ and
$$
  \sum_{\g\in\G\cap I_n}|b_\g|= |\hat\mu^c|(I_n)>5C_5e^{2\pi H|t_n|}.
$$
 If the set $\G\cap (t_n-\eta,t_n+\eta)$ is finite for some $n\in\N$, then we set $A_n=\G\cap (t_n-\eta,t_n+\eta)$. Otherwise, taking into account that
$\hat\mu$ is a measure and $|\hat\mu|(t_n-\eta,t_n+\eta)<\infty$, we
choose $A_n$ to be a finite subset of $\G\cap (t_n-\eta,t_n+\eta)$ such that
\begin{equation}\label{sm}
  \sum_{\g\in\G\cap (t_n-\eta,t_n+\eta)\setminus A_n}|b_{\g}|<C_5e^{2\pi H|t_n|}.
\end{equation}
With this setting, we have
$$
 \sum_{\g\in A_n\cap I_n}|b_{\g}|>4C_5e^{2\pi H|t_n|}.
$$
Using the Kroneker lemma, we can find $s_n\in\R$ and  $m_\g^{(n)}\in\Z$ such that for all $\g\in A_n$
$$
|s_n\g+\frac{\arg b_{\g}}{2\pi}-m_\g^{(n)}|<\frac{1}{6}.
$$
Therefore, for all $\g\in A_n$ we have
$$
  \Re e^{2\pi is_n\g}b_\g=|b_\g|\cos(2\pi s_n\g+\arg b_\g)> \frac{|b_\g|}{2} >0,
$$
and
\begin{equation}\label{arg}
\Re\sum_{\g\in A_n\cap I_n}e^{2\pi is_n\g}b_\g>2C_5e^{2\pi H|t_n|}.
\end{equation}
Let $\p\in\D$ be the function defined at the beginning of the proof.
It follows from \eqref{sm}, \eqref{arg} and the inequality $\Re e^{2\pi is_n\g}b_\g>0$ for all $\g\in A_n$ that
\begin{gather*}\label{c1}
   \left|\int e^{2\pi is_nt}\p(t-t_n)\hat\mu^c(dt)\right| \ge\left|\sum_{\g\in A_n}\p(\g-t_n)e^{2\pi is_n\g}b_\g\right|-\sum_{\g\in\G\cap(t_n-\eta,t_n+\eta)\setminus A_n}|b_{\g}|\\
   \ge \Re\sum_{\g\in A_n\cap I_n}e^{2\pi is_n\g}b_\g+
   \Re\sum_{\g\in A_n\setminus I_n}\p(\g-t_n)e^{2\pi is_n\g}b_\g-C_5e^{2\pi H|t_n|}> C_5e^{2\pi H|t_n|}.
\end{gather*}
 The $c$-Fourier transform of the function $ e^{2\pi is_n t}\p(t-t_n)$ equals  $e^{-2\pi it_n(z-s_n)}\hat\p^c(z-s_n)$. Hence, by definition of the $c$-Fourier transform of the measure,
$$
   \int_\R e^{2\pi is_nt}\p(t-t_n)\hat\mu^c(dt)=
     \int_{\SH} e^{-2\pi it_n(z-s_n)}\hat\p^c(z-s_n)\mu(dz) =   \int_{\SH} e^{-2\pi it_nz}\hat\p^c(z)\mu_{-s_n}(dz).
$$
By  \eqref{p7}, the modulus of the last integral does not exceed $C_5 e^{2\pi|t_n|H}$. This bound  contradicts what was before, hence \eqref{pr1} is true.
 Arguing as in the proof of Theorem \ref{T2}, we obtain that the measure $\hat\mu^c$ has the type  with respect to the exponential growth at most $2\pi H$. \qquad $\Box$


\bigskip
\begin{thebibliography}{19}


\bibitem{AKV} Alon, L., Kummer, M., Kurasov, P., Vinzant, C. \emph{Higher dimensional Fourier Quasicrystals from
Lee-Yang varieties,}  Invent. Math. {\bf 239} (2025), 321--376. https://doi.org/10.1007/s00222-024-01307-8

\bibitem{B} Bohr, H. Almost Periodic Functions, ed. Chelsea, New-York, 1951, 114pp. https://www.amazon.co.uk/Almost-Periodic-Functions-Chelsea-Publishing/dp/082840027X

\bibitem{BF} Boyvalenkov, P., Favorov, S.Yu. \emph{Growth of masses of crystalline measures,} Studia Mathematica 290 (2) (2026), 109-119.  https:// DOI: 10.4064/sm250325-8-1,

\bibitem{F5} Favorov, S.Yu. \emph{The crystalline measure that is not a Fourier Quasicrystal,}  Analysis Mathematica {\bf 50} (2024), 455--462. https://DOI:10.1007/s10476-024-00031-y

\bibitem{F6} Favorov, S.Yu. \emph{Non-negative crystalline and Poisson measures in the Euclidean space,} Studia Mathematica {\bf 278} (2024), 81--98. https://DOI: 10.4064/sm240507-2-8

\bibitem{F7} Favorov, S.Yu. \emph{Analogues of Fourier quasicrystals for a strip,}  Analysis Mathematica, {\bf 51}, no.2, 457-475 https://doi.org/10.1007/s10476-024-00089-2

 \bibitem{F9} Favorov, S.Yu. \emph{Application of methods of quasicrystals theory to entire functions of exponential growth,}  Anal. Math. {\bf 51}, no.4 (2025), 1313-1325 https://doi.org/10.1007/s10476-025-00106-

\bibitem{G} Goncalves, F. \emph{A classification of Fourier summation formulas and crystalline measures,} https://arXiv.org/pdf/2312.11185

\bibitem{GV} Goncalves, F.,  Vedana, G. \emph{A complete classification of Fourier summation formulas on the real line,}  (2023), https://arxiv.org/pdf/2504.02741

\bibitem{KS} Kurasov, P., Sarnak, P. \emph{Stable polynomials and crystalline measures,} J. Math. Phys. {\bf 61} no.8  (2020), https://doi.org/10.1063/5.0012286

\bibitem{LT} Lawton, W. M., Tsikh, A. K. \emph{Fourier quasicrystals on Rn,} The Journal of Geometric Analysis,  art. 93, 50 pp., (2025),  https://doi.org/10.1007/s12220-025-01911-x

\bibitem{L} Levin, B.Ja. Lections on Entire Functions. Transl. of Math. Monograph, MMONO/150, AMS Providence, R1, 1996, 248pp.   https://bookstore.ams.org/view?ProductCode=MMONO/150

\bibitem{Le} Levitan, B.M. Almost Periodic Functions. Gostehizdat, Moscow. 1953, 396pp. (In Russian).

\bibitem{M3} Meyer, Y. \emph{Global and local estimates on trigonometric sums,} Trans. R. Norw. Soc. Sci. Lett. no.2 (2018), 1--25. https://www.dknvs.no/wp-content/uploads/2018/12/Skrifter-2-2018-Meyer-Scr.pdf

\bibitem{M4} Meyer, Y. \emph{Multidimensional crystalline measures,} Trans. R. Norw. Soc. Sci. Lett. no.1 (2023), 1--24. https://www.dknvs.no/wp-content/uploads/2023/07/Skrifter-01-2023-Yves-Meyer.pdf

\bibitem{OU} Olevskii, A., Ulanovskii A. \emph{Fourier quasicrystals with unit masses,} Comptes Rendus Mathematique, {\bf 358} (2020), 1207--1211.https://doi.org/10.5802/crmath.142Harmonic Analysis

\bibitem{RV} Radchenko, D., Viazovska, M. \emph{Fourier interpolation on the real line,} Publ. Math. IHES {\bf 129} (2019), 51--81. https://doi.org/10.1007/s10240-018-0101-z

\end{thebibliography}
\end{document}